\documentclass[12pt]{article}

\usepackage{makeidx,epsfig,lscape}
\usepackage{graphicx} 
\usepackage[T1,T2A]{fontenc}
\usepackage[utf8]{inputenc}
\usepackage{latexsym}
\usepackage{amsmath,amsthm,amssymb}  
\usepackage{mathrsfs}  
\usepackage{bm}         
\usepackage{accents}
\usepackage{subfigure}
\usepackage[font=small,format=plain,labelfont=bf,up,textfont=it,up]{caption}
\usepackage[english]{babel}
\usepackage{latexsym}
\usepackage{xcolor}
\usepackage{subfigure}
\usepackage{environ}
\NewEnviron{myequation}{%
	\begin{equation}
		\scalebox{0.95}{$\BODY$}
	\end{equation}
}
\NewEnviron{myequation2}{%
	\begin{equation}
		\scalebox{0.9}{$\BODY$}
	\end{equation}
}

\NewEnviron{myequation3}{%
	\begin{equation}
		\scalebox{0.875}{$\BODY$}
	\end{equation}
}

\usepackage{environ}

\newtheorem{theorem}{Theorem}

\newtheorem{lemma}[theorem]{Lemma}
\newtheorem{remark}{Remark}

\newcommand{\D}{\frac{\partial}{\partial D}}
\newcommand{\be}{\begin{equation}}
\newcommand{\ee}{\end{equation}}
\usepackage{graphics}

\begin{document}
\def \R{{\mathbb R}}
\def \C{{\mathbb C}}
\def \Cn{{\mathbb C}^n}

\def \bz{\mbox{\boldmath $z$}}
\def \sbz {\mbox{\boldmath $\scriptstyle z$}}
\def \bxi {\mbox{\boldmath $\xi$}}
\def \bzeta {\mbox{\boldmath $\zeta$}}
\def \sbzeta {\mbox{\boldmath $\scriptstyle\zeta$}}
\def \ba {\mbox{\boldmath $\alpha$}}
\def \D {D}
\def \cD {\partial{D}}
\def \bW {\bar{W}}
\def \cW {\partial{W}}
\def \mut {\tilde{\mu}}
\renewcommand{\i}{\mathsf{i}}

\def \P{\mathscr{P}}
\def \Pp{\mathcal{P}}
\def \be{\begin{equation}}
	\def \ee{\end{equation}}

\title{Analyticity domains of critical points of polynomials. A proof of Sendov's conjecture}
\author{Petar P. Petrov\thanks{Email: peynovp@gmail.com}}
\date{\today}
\maketitle

\begin{abstract}
	Let $\P_{n}^c(\bar{\mu},\bar{\nu})$ be the set of all complex polynomials  $p(z)=\prod_{i=1}^{m}(z-z_i)^{\mu_i}$, $\sum_{i=1}^m\mu_i=n$, with derivatives of the form  
	$$
	p'(z)=n\prod_{i=1}^{m}(z-z_i)^{\mu_i-1}\prod_{j=1}^{k}(z-\xi_j)^{\nu_j}, ~\sum_{j=1}^k\nu_j=m-1.
	$$
	In this note we prove the following:\par\medskip
	{\it
		\noindent For a fixed ordering $\alpha=(1,2,\ldots,m)$,   the distinct zeros $\{z_i\}_{i=1}^m$ and the distinct critical points   of the second kind $\{\xi_j\}_{j=1}^k$ of polynomials from $\P_{n}^c(\bar{\mu},\bar{\nu})$ are analytic  functions  $\{z_i^{\alpha\beta}\}_{i=1}^m$  and $\{\xi_j^{\alpha\beta}\}_{j=1}^k$,  resp.,   $\beta=(i_1,i_2,\ldots,i_{k+1})$, of any of the variables $(z_{i_1},z_{i_2},\ldots,z_{i_{k+1}})$ in the domain 
		$$
		\{(z_{i_1},z_{i_2},\ldots,z_{i_{k+1}})\in \C^{k+1}~\vert~p\in \P_{n}^c(\bar{\mu},\bar{\nu}) \},
		$$
		being also continuous on its boundary.}\par\medskip\noindent
	This statement gives an immediate proof to the well-known conjecture of Bl. Sendov \cite{sen}: \par\medskip {\it 
		\noindent If $n\ge 2$ and $p(z)=\prod_{i=1}^n (z-z_i)$ is a polynomial  of degree $n$  such  that $z_i\in \C$, $\vert z_i\vert\le 1$, $i=1,2,\ldots,n$, then for every $i=1,2,\ldots,n$, the disk $\{z\in \C\,|\,\vert z_i-z\vert \le 1\}$ contains at least one zero of $p'(z)$.}
\end{abstract}
\section{Domains of analyticity}
Let $\P_n^c$ be the set of complex polynomials of the form $p(z)=\prod_{i=1}^n(z-z_i)$.
For given multiplicities $\bar{\mu}:=(\mu_1,\ldots,\mu_m)$, $m\ge 1$, and $\bar{\nu}:=(\nu_1,\ldots,\nu_k)$, $k\ge 1$, such that
\be\label{multi}
\sum_{i=1}^m \mu_{i}=n \quad \hbox{ and} \quad
\sum_{i=1}^k \nu_{i}=m-1,
\ee
let $\C^n(\bar{\mu},\bar{\nu})\subset \C^n$ be the set consisting of all  $\bar{z}=((z_1,\mu_1),\ldots,(z_m,\mu_m))$, $z_i\neq z_j$, $1\le i<j\le m$, such that there exists a polynomial $p\in \P_n^c$ satisfying
\par\smallskip\noindent
(i)  $p(z)=p(\bar{z};z)=\prod_{i=1}^{m}(z-z_{i})^{\mu_{i}}$;
\par\noindent
(ii)  $p$ possesses exactly $k$ distinct critical points of the second kind (that is, zeros of $p'$ which are not zeros of $p$) with multiplicities $\nu_1,\nu_2,\ldots,\nu_k$ and
$$
p'(z)=n\prod_{i=1}^{m}(z-z_{i})^{\mu_{i}-1}\prod_{j=1}^{k}(z-\xi_{j})^{\nu_{j}}.
$$
\par\smallskip\noindent
The set of polynomials $p$ satisfying (i) and (ii) will be denoted by 
$$
\P_{n}^c(\bar{\mu},\bar{\nu}):=\left \{p(\bar{z};z)\in \P_n^c~\vert~ \bar{z} \in \C^n(\bar{\mu},\bar{\nu})\right \}.
$$ 
\par\medskip
\noindent
\begin{theorem}\label{theorem1}
	\noindent For a prescribed ordering $\alpha=(1,2,\ldots,m)$,   the distinct zeros $\{z_i\}_{i=1}^m$ and the distinct critical points   of the second kind $\{\xi_j\}_{j=1}^k$ of polynomials from $\P_{n}^c(\bar{\mu},\bar{\nu})$ are analytic  functions  $\{z_i^{\alpha\beta}\}_{i=1}^m$  and $\{\xi_j^{\alpha\beta}\}_{j=1}^k$,  respectively,   $\beta=(i_1,i_2,\ldots,i_{k+1})$, of any of the variables $(z_{i_1},z_{i_2},\ldots,z_{i_{k+1}})$ in the domain 
	$$
	\{(z_{i_1},z_{i_2},\ldots,z_{i_{k+1}})\in \C^{k+1}~\vert~\bar{z}\in \C_{n}(\bar{\mu},\bar{\nu}) \},
	$$
	being also continuous on its boundary.
\end{theorem}
\par\medskip
\noindent
The proof of Theorem~\ref{theorem1} is entirely based on the following lemma.
\par\medskip
\noindent
\begin{lemma}\label{lemma1} 
	Let there be given multiplicities $(\bar{\mu},\bar{\nu})$ satisfying (\ref{multi}) with $m>1$, and let a polynomial $p$ with zeros $\{(z_i,\mu_i)\}_{i=1}^m$ belong to $\P_{n}^c(\bar{\mu},\bar{\nu})$.  Then, the distinct critical points of the second kind  $\{\xi_j\}_{j=1}^{k}$  and the first $m-k-1$ zeros $z_1,z_2,\ldots,z_{m-k-1}$, are (locally) analytic functions of the remaining ones $z_{m-k},z_{m-k+1},\ldots,z_m$.
\end{lemma}
\begin{proof}
	Let  $p_0(\bar{z}^0;z)\in \P_{n}^c(\bar{\mu},\bar{\nu})$.
	Set $s:=m-1-k$. Clearly, $s\ge 0$ (with equality only when $\nu_j=1,~j=1,2,\ldots,k$).
	Let us consider the following system of equations
	\be\label{system}
	p^{(\ell)}(\xi_j)=0,\quad j=1,2,\ldots,k;~\ell=1,2,\ldots,\nu_j,
	\ee
	for $p=p_0(\bar{z}^0;z)$.
	The proof of the lemma is based on the system (\ref{system}) and the Implicit Mapping Theorem (see, e.g., \cite[p.\,28]{kaup}.
	We need to show that there exist analytic functions
	$z_1=z_1(z_{s+1},\ldots z_m),\ldots,z_s=z_s(z_{s+1},\ldots z_m)$ and $\xi_1=\xi_1(z_{s+1},\ldots z_m),\ldots,\xi_k=\xi_k(z_{s+1},\ldots z_m)$,  satisfying (\ref{system}) in an open neighbourhood 
	of $(z^0_{s+1},\ldots, z^0_m)$, or in other words, that the system of $m-1$ equations (\ref{system}) can be locally and analytically solved with respect to the $m-1$ functions $\{z_i\}_{i=1}^s$, $\{\xi_j\}_{j=1}^k$, or equivalently, that the Jacobian matrix $J$ of (\ref{system}) possesses the maximal rank. 
	The Jacobian matrix of (\ref{system}) (we omit for simplicity the upper index 0) is $J=\hbox{diag\,}\{p''(\xi_{1}),\ldots,p''(\xi_{k})\}$ if $s=0$, and
	\begin{myequation2}\nonumber
		J=\left [ 
		\begin{array}{ccccccccc}
			\Omega_1'(\xi_{1}) & \Omega_1''(\xi_1) &\ldots & 
			\Omega_1^{(\nu_{1})}(\xi_{1}) & \ldots &\Omega_1'(\xi_{k}) & \Omega_1''(\xi_{k}) &\ldots &  \Omega_1^{(\nu_{k})}(\xi_{k})\\
			\vdots & \cdots & \vdots & \vdots  &\cdots &\vdots &\cdots &\vdots & \vdots\\
			\Omega_s'(\xi_{1}) & \Omega_s''(\xi_{1}) & \ldots & 
			\Omega_s^{(\nu_{1})}(\xi_{1}) & \ldots & \Omega_s'(\xi_{k}) & \Omega_s''(\xi_{k}) &\ldots &  \Omega_s^{(\nu_{k})}(\xi_{k})\\
			0 & \ldots & 0 &
			p^{(\nu_{1}+1)}(\xi_{1}) &  \ldots  & 0 & \ldots & 0 & 0\\
			0 & \ldots & 0 &
			0 &  \ldots  & 0 & \ldots & 0 & 0\\
			\vdots & \cdots & \vdots & \vdots  &\cdots &\vdots &\cdots &\vdots & \vdots\\
			0 & \ldots & 0 &
			0 &  \ldots  & 0 & \ldots & 0 & 0\\
			0 & \ldots & 0 &
			0 &  \ldots & 0 & \ldots & 0 & p^{(\nu_{k}+1)}(\xi_{k})
		\end{array}
		\right ]^{\hbox{T}},
	\end{myequation2}
	where $\Omega_i(z)=-\frac{\mu_{i}p(z)}{z-z_{i}},~i=1,2,\ldots,s$, if $s\ge 1$.
	We need to prove that $\hbox{rank}\,J=m-1$. If $s=0$, this is a direct consequence of $p''(\xi_j)\neq 0$, $j=1,2,\ldots ,k$. Let $s\ge 1$ and let us assume the contrary, that is, $\hbox{rank}\,J<m-1$. Using that $p^{(\nu_{j}+1)}(\xi_j)\neq 0$, $j=1,2,\ldots ,k$, our assumption implies the existence of a non-zero vector
	$(\beta_1,\beta_2,\ldots,\beta_s)\in \C^s$ such that the polynomial $\Omega\in \Pp_{n-1}^c$ (the set of all complex polynomials of degree at most $n-1$),
	$$
	\Omega(z):=-{\beta_1}{\mu_{1}^{-1}} \Omega_1(z)-{\beta_2}{\mu_{2}^{-1}}\Omega_2(z)-\cdots -{\beta_s}{\mu_{s}^{-1}} \Omega_s(z)
	=\left (\sum_{i=1}^s\frac{\beta_i}{z-z_{i}}\right )p(z),
	$$
	satisfies
	\be\label{omega}
	\Omega^{(\ell)}(\xi_{j})=\left (\sum_{i=1}^s\frac{\beta_{i}}{z-z_{i}}\right )_{\big\vert z=\xi_{j}}^{(\ell)}p(\xi_{j})=0,\quad j=1,2,\ldots,k;~\ell=1,2,\ldots,\nu_{j}-1.
	\ee
	Denote 
	$$
	\omega(z):=\prod_{i=1}^s(z-z_{i}),\quad \omega_i(z):=\frac{\omega(z)}{z-z_{i}},~i=1,2,\ldots,s,
	$$
	and
	$$
	q(z):=\prod_{j=1}^{k}(z-\xi_{j})^{\nu_{j}-1}.
	$$
	Clearly, $q(z)\in \P_s^c$  since $\sum_{j=1}^{k}(\nu_{j}-1)=m-1-k=s$.
	Using that $z_{i}\neq \xi_{j}$, it easily follows from (\ref{omega}) that there exists a vector
	$(\alpha_1,\alpha_2,\ldots,\alpha_s)\in \C^s$ such that
	\be\label{omega1}
	q(z)\sum_{i=1}^s\alpha_i\omega_i(z)+\sum_{i=1}^s\beta_i\omega_i^2(z)\equiv 0.
	\ee
	If $\beta_i\neq 0$, $i=1,2,\ldots,\tilde{s}$, and $\beta_i= 0$, $i=\tilde{s}+1,\tilde{s}+2,\ldots,s$, $1\le \tilde{s}<s$, then we rewrite (\ref{omega1}) in this way
	$$
	q_1(z)q_2(z)\frac{\sum_{i=1}^s\alpha_i\omega_i(z)}{\prod_{i=\tilde{s}+1}^s (z-z_i)^2}+\sum_{i=1}^{\tilde{s}}\beta_i\tilde{\omega}_i^2(z)\equiv 0,
	$$
	where $q_1(z)\in \P_{\tilde{s}}^c$, $q_2(z)\sum_{i=1}^s\alpha_i\omega_i(z)/\prod_{i=\tilde{s}+1}^s (z-z_i)^2\in \Pp_{\tilde{s}-1}^c$, $\tilde{\omega}_i(z)$ $:=\prod_{\substack{j=1\\ j\neq i}}^{\tilde{s}}(z-z_i)$, and arrive at a similar equation with $s$ replaced by $\tilde{s}$
	$$
	q_1(z)\sum_{i=1}^{\tilde{s}}\tilde{\alpha}_i\tilde{\omega}_i(z)+\sum_{i=1}^{\tilde{s}}\beta_i\tilde{\omega}_i^2(z)\equiv 0.
	$$
	Therefore, without loss of generality, we can and will assume that  $\beta_i\neq 0$, $i=1,2,\ldots,s$. 
	Setting in (\ref{omega1})  $z=z_{i}$, $i=1,2,\ldots,s$,  we get
	$$
	\beta_i=-\alpha_i q(z_{i})[\omega_i(z_{i})]^{-1},\quad i=1,2,\ldots,s,
	$$
	and  
	\be\label{omega2}
	q(z)\sum_{i=1}^s\alpha_i\omega_i(z)-\sum_{i=1}^s\alpha_iq(z_{i})[\omega_i(z_{i})]^{-1}\omega_i^2(z)\equiv 0.
	\ee  
	It follows from (\ref{omega2}) that
	\be\label{omega21}
	\sum_{i=1}^s\alpha_i=0
	\ee
	and
	\be\label{omega22}
	q(z_{i})\sum_{\substack{j=1\\ j\neq i}}^s\alpha_j\omega_j'(z_{i})+\alpha_i[q'(z_{i})\omega_i(z_{i})-q(z_{i})\omega_i'(z_{i})]=0,\quad i=1,2,\ldots,s.
	\ee
	According to Lagrange interpolation formula
	$$
	q(z)=\sum_{i=1}^sq(z_i)\frac{\omega_i(z)}{\omega_i(z_{i})}+\omega(z),
	$$
	and in addition, by (\ref{omega2}),   
	\be\label{omega3}
	\omega (z)\sum_{i=1}^s\alpha_i\omega_i(z)+\sum_{i=1}^s\sum_{\substack{j=1\\ j\neq i}}^s\left [\alpha_i\frac{q(z_{j})}{\omega_j(z_{j})}+\alpha_j\frac{q(z_{i})}{\omega_i(z_{i})}\right ]\omega_i(z)\omega_j(z)\equiv 0.
	\ee  
	Dividing (\ref{omega3}) by $\omega(z)$ and setting $z=z_{i}$, $i=1,2,\ldots, s$, give
	\be\label{omega4}
	\alpha_i\left [1+\sum_{\substack{j=1\\ j\neq i}}^s\frac{q(z_{j})}{\omega_j(z_{j})(z_{i}-z_{j})}\right ]
	+\frac{q(z_{i})}{\omega_i(z_{i})}\sum_{\substack{j=1\\ j\neq i}}^s\frac{\alpha_j}{z_{i}-z_{j}}=0,\quad i=1,2,\ldots,s.
	\ee
	The first sum in (\ref{omega4}) will be calculated from the following equation, using the properties of the divided difference \begin{eqnarray*}
		1&=&q[z_{1},\ldots,z_{i-1},z_{i},z_{i},z_{i+1},\ldots,z_{s}]\\
		&=&
		-\sum_{\substack{j=1\\ j\neq i}}^s\frac{q(z_{j})}{\omega_j(z_{j})(z_{i}-z_{j})}+\frac{q'(z_{i})}{\omega_i(z_{i})}+\frac{q(z_{i})}{\omega_i(z_{i})}\cdot\frac{\omega_i'(z_{i})}{\omega_i(z_{i})},
	\end{eqnarray*}
	or more precisely,
	\be\label{omega5}
	\alpha_i\left [q'(z_{i})+q(z_{i})\cdot \frac{\omega_i'(z_{i})}{\omega_i(z_{i})} \right ]
	+q(z_{i})\sum_{\substack{j=1\\ j\neq i}}^s\frac{\alpha_j}{z_{i}-z_{j}}=0,\quad i=1,2,\ldots,s.
	\ee 
	Finally, we get from (\ref{omega22}) and  (\ref{omega5})
	$$
	2\alpha_iq(z_{i})\omega_i'(z_{i})=\alpha_iq(z_{i})\omega''(z_{i})=0,\quad i=1,2,\ldots,s.
	$$
	Since $\alpha_iq(z_{i})\neq 0$, $i=1,2,\ldots,s$, this means
	that $\omega\in \P_1^c$, that is, $s=1$, and by (\ref{omega21}),
	$\alpha_1=\beta_1=0$. This is a contradiction to our assumption that $\hbox{rank}\,J<m-1$ and hence, $\hbox{rank}\,J=m-1$. The lemma is proved. 
\end{proof}
\par\medskip
\noindent
\section{Sendov's conjecture}
Sendov's conjecture is one of the fundamental problems in the theory of complex polynomials. It was announced back in 1958 by Bulgarian mathematicians  Blagovest Sendov and Lyubomir Iliev (see \cite{sen} for the history of the problem and for references before 2001). Despite  the significant analytical and computational efforts within the last more than 60 years,  there has not been yet a definitive proof of its validity. Here we show that it is an almost direct consequence of Theorem~\ref{theorem1}.
\par
Let $D:=\{z\in \C\,|\,\vert z\vert \le 1\}$ be the closed unit disk in $\C$. 
Sendov's conjecture states:
\par\medskip\noindent
{\bf Conjecture 1.} If $n\ge 2$ and $p(z)$ is a polynomial from   $\P_n^c$ such that $\{z_i\}_{i=1}^n\subset D$, then for every $i=1,2,\ldots n$, the disk $\{z\in \C\,|\,\vert z_i-z\vert \le 1\}$ contains at least one zero of $p'(z)$.
\par\medskip\noindent
Set $\bar{z}:=(z_1,z_2,\ldots, z_n)$.
If $p'(z)=n\prod_{j=1}^{n-1}(z-\zeta_j)$, then the statement of Conjecture~1 is equivalent to
$$
\max_{\bar{z}\in D^n}\max_{1\le i\le n} \min_{1\le j\le n-1}\vert z_i-\zeta_j\vert=1.
$$
In what follows, the zeros of $p$ will also be represented in the (already familiar) form  $\{(z_i,\mu_i)\}_{i=1}^m$   using for simplicity the same letters. Let 
$$\C^n(m,k):=\cup \left\{\C^n(\bar{\mu}',\bar{\nu}')~|~ m'=m,~k'=k\right\}.$$
\begin{remark}\label{remark1}
	Taking Theorem~\ref{theorem1} into account, it follows that, up to the ordering $\alpha=(1,2,\ldots,m)$ of the variables, there exist unique functions $\{z^\alpha_i\}_{i=1}^m$  and  $\{\xi^\alpha_j\}_{j=1}^k$ of $\{z_i\}_{i=1}^{k+1}$, analytic on $\C^n(m,k)$, which naturally coincide into new functions on $\C^n(m',k')$, $m'\le m$, $k'<k$, when some of the variables (zeros of $p$) and (or) some of the functions $\{\xi^\alpha_j\}_{j=1}^k$ (critical points of the second kind of $p$) overlap. Moreover, $\{z^\alpha_i\}_{i=1}^m$  and  $\{\xi^\alpha_j\}_{j=1}^k$ are continuous on its boundary 
	$$
	\partial \, \C^n(m,k)=\C^n(m-1,k-1)\cup \C^n(m,k-1),
	$$
	$2\le m\le n,~2\le k\le m-1$.
\end{remark} 
\par
\bigskip
\noindent {\it Proof of Conjecture 1.}~
Let $\bar{z}\in \C^n(m,k)$. Define the functions
\begin{equation}\label{SS}
	S(\bar{z}):=\max_{1\le i\le n} \min_{1\le j\le n-1}\vert z_i-\zeta_j\vert
\end{equation}
and
\begin{equation}\label{Sell}
	S_{\ell}:=\min_{1\le j \le k}\vert z_\ell-\xi_j\vert
\end{equation}
for all $1\le \ell \le m$ such that $\mu_\ell=1$. 
We will prove the following enhancement of the conjecture under consideration:
\par\medskip\noindent
\begin{theorem}\label{theorem2}  If $\bar{z}\in D^m$ and $m\ge 2$, then we have
	$$
	\min_{1\le j\le k} \vert z_\ell-\xi_j\vert\le 1,
	$$
	for any $z_\ell$, $\mu_\ell=1$, with equality if and only if $m=n$, $k=1$ and $\{z_i\}_{i=1}^n$ coincide with the roots of unity $ \{e^{\frac{2(i-1)\pi \i}{n}} \}_{i=1}^{n}$.
\end{theorem}
\par\medskip\noindent
{\it Proof.}~Let us show first that the proof can be reduced to the case $k=1$. Let $k\ge 2$, $\bar{z}\in C^n(m,k)$ and $\vert z_\ell\vert<1$. 
Denote by $\hbox{rad}_c(A)$ the radius of the circumself disk of a set of points $A$ (the smallest closed disk containing $A$). Let $\Phi$ be the map from Theorem~\ref{theorem1} and consider all sufficiently small local paths $\bar{z}(t),~t\in[0,t_0)$ such that $z_\ell(t)\in \hbox{int}(D),~t\in [0,t_0)$. For each such path we take the configuration $\Phi(\bar{z}(t))$ and consider the points
$\Phi(\bar{z}(t))/\hbox{rad}_c(t)\subset D^{m+k}$. All these points constitute a set
$V_1\times \ldots \times V_m\times W_1\times\ldots \times W_k\subset D^{m+k}$ such that $V_\ell$ and $W_1,\ldots ,W_k$ are open neighbourhoods of $z_\ell$ and $\{\xi_j\}_{j=1}^k$, respectively. This means that we can move the points  $z_\ell,\xi_1,\xi_2,\ldots,\xi_k$ so that to increase $S_\ell$ and $\bar{z}$ to remain in $\C^n(m,k)\cap D^{m}$. Therefore,  the maximum point of $S_\ell$ is located on $\partial(\C^n(m,k))\cap D^{m}$.  Now, all we need to do next is to consider the case $k=1$.
\par\medskip
\noindent
\textbf{\emph{The case $ \mathbf{{k=1}}$.}} Let $(\bar{\mu},\bar{\nu})$ be such that $k=1$ and  $\nu_1=m-1$, respectively. Let  $p(z)=\prod _{i=1}^m (z-z_i)^{\mu_i}$ and  $p'(z)=n(z-\xi_1)^{m-1}\prod _{i=1}^m (z-z_i)^{\mu_i-1}$.
Then, we have consecutively 
$$
\frac{p'(z)}{p(z)}=\frac{n(z-\xi_1)^{m-1}}{\prod _{i=1}^m (z-z_i)}=\sum_{i=1}^m\frac{\mu_i}{z-z_i},
$$
$$
(z-\xi_1)^{m-1}=\frac{1}{n}\sum_{i=1}^m\mu_i\omega_i(z),\quad \omega_i(z)=\prod_{\substack{j=1\\ j\neq i}}^m(z-z_j),
$$
$$
\xi:=\xi_1=\sum_{i=1}^m\frac{\mut_i}{n}z_i,\quad \mut_i:=\frac{n-\mu_i}{m-1},~i=1,2,\ldots,m.
$$
We have to prove that $\vert z_\ell -\xi\vert \le 1$ for all $z_\ell$ such that $\mu_\ell=1$.
Let 
$$
\vert z_{i_0} -\xi\vert = \max_{\mu_\ell=1} \vert z_\ell -\xi\vert,~\hbox{for some}~z_{i_0},~\vert z_{i_0} \vert <1.
$$ 
Without loss of generality, we will assume that $\hbox{rad}_c(\bar{z})=1$. Let also $m\ge 9$ (Conjecture 1 is proved for polynomials with $m\le 8$ \cite{bx}). 
Denote 
$$
\xi=:a+\i b,~ z_i=:a_i+\i b_i,~\xi-z_i=:c_ie^{\i\theta_i},~c_i>0,~\theta_i\in [0,2\pi),
$$
for $i=1,2,\ldots,m$, and $c:=\prod_{i=1}^mc_i$, $\theta:=\sum_{i=1}^m\theta_i$.
By means of Kuhn-Tucker necessary conditions, we will investigate the extremal points of the problem
\be\label{extrpr}
F_0\rightarrow \max,\quad f(z_1,z_2,\ldots ,z_m;z)\equiv 0,~a_i^2+b_i^2\le 1,~i=1,2,\ldots,m,
\ee
where
$$
F_0:=F_0(z_1,z_2,\ldots ,z_m):=(a-a_{i_0})^2+(b-b_{i_0})^2,
$$
and
$$
f:=f(z):=f(z_1,z_2,\ldots ,z_m;z):=(z-\xi)^{m-1}-\frac{1}{n}\sum_{i=1}^m\mu_i\omega_i(z).
$$
To this aim, set
$$
F:=F(z):=F_0(z_1,z_2,\ldots ,z_m)-
\lambda_1\Re f(z)-\lambda_2\Im f(z)-\sum_{i=1}^m\eta_i(a_i^2+b_i^2-1),
$$
where $\eta_i\ge 0$ and  $\eta_i(a_i^2+b_i^2-1)=0$, $i=1,2,\ldots,m$, at the extremal points. \par\noindent
The construction of the function $F$ needs some additional comments.
Formally, we should define it by using  $m$ different equality constraints, corresponding to $m$ different arbitrary points $z$ in an open neighbourhood of $\xi$. However, our calculations below implicitly use these constraints in taking the partial derivatives of $F$
with respect to $a,~b$ and $\{a_i,b_i\}_{i=1}^m$ and setting $z=\xi$. This is justified by the fact that the constraint functions (and so, the extremal points) are continuous functions of $z$. 
Next,  we  calculate the corresponding partial derivatives at $z=\xi$. First,
$$
\frac{\partial f}{\partial a_i}=-\frac{\mut_i}{n}(m-1)(z-\xi)^{m-2}+\frac{1}{n}\sum_{\substack{j=1\\ j\neq i}}^m\mu_j\omega_{ij}(z),~\omega_{ij}(z)=\frac{\omega_i(z)}{z-z_j},~i\neq j.
$$
From here,
$$
(z-z_i)\frac{\partial f}{\partial a_i}=-\frac{\mut_i}{n}(m-1)(z-\xi)^{m-2}(z-z_i)+(z-\xi)^{m-1}-\frac{\mu_i}{n}\omega_i(z)
$$
and
$$
(\xi-z_i)\frac{\partial f}{\partial a_i}(\xi)=-\frac{\mu_i}{n}\omega_i(\xi).
$$
Consequently,
\begin{eqnarray*}
	c_ie^{\i\theta_i}\frac{\partial f}{\partial a_i}(\xi)&=& -\frac{c\mu_i}{nc_i}e^{\i (\theta-\theta_i)},\qquad \frac{\partial f}{\partial a_i}(\xi)= -\frac{c\mu_i}{nc_i^2}e^{\i (\theta-2\theta_i)},\\
	c_ie^{\i\theta_i}\frac{\partial f}{\partial b_i}(\xi)&=& -\frac{c\mu_i}{nc_i}\i e^{\i (\theta-\theta_i)},\qquad \frac{\partial f}{\partial b_i}(\xi)= -\frac{c\mu_i}{nc_i^2}\i e^{\i (\theta-2\theta_i)},
\end{eqnarray*}
and
\begin{eqnarray*}
	\frac{\partial \Re f}{\partial a_i}(\xi)&=& -\frac{c\mu_i}{nc_i^2}\cos (\theta-2\theta_i),\qquad \frac{\partial \Re f}{\partial b_i}(\xi)= +\frac{c\mu_i}{nc_i^2}\sin (\theta-2\theta_i)\\
	\frac{\partial \Im f}{\partial a_i}(\xi)&=& -\frac{c\mu_i}{nc_i^2}\sin (\theta-2\theta_i),\qquad \frac{\partial \Im f}{\partial b_i}(\xi)= -\frac{c\mu_i}{nc_i^2}\cos (\theta-2\theta_i).
\end{eqnarray*}
For the partial derivatives of $F_0$ with respect to $a_i,b_i$, $i\neq i_0$, we have
\begin{eqnarray*}
	\frac{\partial F_0}{\partial a_i}&=&\frac{2\mut_i}{n}(a-a_{i_0})=\frac{2\mut_i}{n}c_{i_0}\cos(\theta_{i_0}),\\
	\frac{\partial F_0}{\partial b_i}&=&\frac{2\mut_i}{n}(b-b_{i_0})=\frac{2\mut_i}{n}c_{i_0}\sin(\theta_{i_0}),
\end{eqnarray*}
and, in addition,
\begin{eqnarray*}
	\frac{\partial F_0}{\partial a_{i_0}}&=&\frac{2\mut_{i_0}-2n}{n}c_{i_0}\cos(\theta_{i_0}),\\
	\frac{\partial F_0}{\partial b_{i_0}}&=&\frac{2\mut_{i_0}-2n}{n}c_{i_0}\sin(\theta_{i_0}).
\end{eqnarray*}
Now, we are ready to write Kuhn-Tucker necessary conditions  for a maximum of $F_0$  at the points $\{a_i+\i b_i\}_{i=1}^m$ under the conditions given in (\ref{extrpr}) for $z=\xi$ (for simplicity, we do not change the notation for the points of maximum). 
More precisely, the following must be fulfilled 
\begin{myequation3}\label{kt-eq1}
	\begin{aligned}
		\frac{\partial F}{\partial a_i}&=&\frac{2\mut_i}{n}c_{i_0}\cos(\theta_{i_0})+\frac{c\mu_i}{nc_i^2}[+\lambda_1\cos (\theta-2\theta_i)+\lambda_2\sin (\theta-2\theta_i)]-2\eta_ia_i=0,\quad i\neq i_0,\\
		\frac{\partial F}{\partial b_i}&=&\frac{2\mut_i}{n}c_{i_0}\sin(\theta_{i_0})+\frac{c\mu_i}{nc_i^2}[-\lambda_1\sin (\theta-2\theta_i)+\lambda_2\cos (\theta-2\theta_i)]-2\eta_ib_i=0,\quad i\neq i_0,\\
		\frac{\partial F}{\partial a_{i_0}}&=&\frac{2\mut_{i_0}-2n}{n}c_{i_0}\cos(\theta_{i_0})+\frac{c\mu_{i_0}}{nc_{i_0}^2}[+\lambda_1\cos (\theta-2\theta_{i_0})+\lambda_2\sin (\theta-2\theta_{i_0})]-2\eta_{i_0}a_{i_0}=0,\\
		\frac{\partial F}{\partial b_{i_0}}&=&\frac{2\mut_{i_0}-2n}{n}c_{i_0}\sin(\theta_{i_0})+\frac{c\mu_{i_0}}{nc_{i_0}^2}[-\lambda_1\sin (\theta-2\theta_{i_0})+\lambda_2\cos (\theta-2\theta_{i_0})]-2\eta_{i_0}b_{i_0}=0.
	\end{aligned}
\end{myequation3}
Using that $\lambda_1=\lambda\cos(\theta_\lambda)$ and $\lambda_2=\lambda\sin(\theta_\lambda)$ for some $\lambda \ge 0$ and $\theta_\lambda\in [0,2\pi)$, Equations (\ref{kt-eq1}) can be rewritten as follows
\begin{gather}\nonumber
	2\mut_i c_{i_0}\cos(\theta_{i_0})+\lambda\frac{c\mu_i}{c_i^2}\cos(\theta-2\theta_i-\theta_\lambda)-2n\eta_ia_i=0,\quad i\neq i_0,\\
	2\mut_i c_{i_0}\sin(\theta_{i_0})-\lambda\frac{c\mu_i}{c_i^2}\sin(\theta-2\theta_i-\theta_\lambda)-2n\eta_ib_i=0,\quad i\neq i_0,\nonumber\\
	(2\mut_{i_0}-2n) c_{i_0}\cos(\theta_{i_0})+\lambda\frac{c\mu_{i_0}}{c_{i_0}^2}\cos(\theta-2\theta_{i_0}-\theta_\lambda)-2n\eta_{i_0}a_{i_0}=0,\nonumber\\
	(2\mut_{i_0}-2n) c_{i_0}\sin(\theta_{i_0})-\lambda\frac{c\mu_{i_0}}{c_{i_0}^2}\sin(\theta-2\theta_{i_0}-\theta_\lambda)-2n\eta_{i_0}b_{i_0}=0.
	\nonumber\end{gather}
and
\begin{gather}\nonumber
	2\mut_i (\bar{\xi}-\bar{z}_{i_0})+\lambda\frac{c\mu_i}{c_i^2}e^{\i(\theta-2\theta_i-\theta_\lambda)}-2n\eta_i\bar{z}_i=0,\quad i\neq i_0,\\
	(2\mut_{i_0}-2n)(\bar{\xi}-\bar{z}_{i_0})+\lambda\frac{c\mu_{i_0}}{c_{i_0}^2}e^{\i(\theta-2\theta_{i_0}-\theta_\lambda)}-2n\eta_{i_0}\bar{z}_{i_0}=0,
	\nonumber\end{gather}
from where
\begin{gather}\label{kt-eq3}
	\lambda e^{-\i\theta_\lambda}\frac{\prod_{j=1}^m(\xi-z_j)}{(\xi-z_i)^2}+\frac{2\mut_i}{\mu_i}(\bar{\xi}-\bar{z}_{i_0})-2n\frac{\eta_i}{\mu_i}\bar{z}_i=0,\quad i\neq i_0,\\
	\lambda e^{-\i\theta_\lambda}\frac{\prod_{j=1}^m(\xi-z_j)}{(\xi-z_{i_0})^2}+\frac{2\mut_{i_0}-2n}{\mu_{i_0}}(\bar{\xi}-\bar{z}_{i_0})-2n\frac{\eta_{i_0}}{\mu_{i_0}}\bar{z}_{i_0}=0.
	\label{kt-eq4}\end{gather}
Equations (\ref{kt-eq3}) and (\ref{kt-eq4}) yield
\begin{equation}\label{kt-eq2}
	\begin{aligned}
		(\xi-z_i)^2\left[\frac{\mut_i}{\mu_i}(\bar{\xi}-\bar{z}_{i_0})-n\frac{\eta_i}{\mu_i}\bar{z}_i\right ]&=\\
		(\mut_{i_0}-n)(\xi-z_{i_0})^2(\bar{\xi}-\bar{z}_{i_0})&=-\frac{\lambda}{2} e^{-\i\theta_\lambda}\prod_{j=1}^m(\xi-z_j),~i\neq i_0,
	\end{aligned}
\end{equation}
since $\mu_{i_0}=1$ and $\eta_{i_0}=0$.
Remind that $\vert z_{i_0}\vert <1$ and 
$$
\mut_{i_0}=\frac{n-1}{m-1},~~\mut_i=\frac{n-\mu_i}{m-1},~i\neq i_0,
$$
$$
\mut_{i_0}-n=\frac{n-1-(m-1)n}{(m-1)}<0,~~\frac{\mut_{i}}{\mu_{i}}=\frac{n-\mu_i}{(m-1)\mu_i},~i\neq i_0.
$$
Let there exist another zero $z_i$ with $\vert z_{i}\vert < 1$ and $\eta_{i}=0$. Then, (\ref{kt-eq2}) implies for $m\ge 4$ 
$$
\frac{\vert \xi-z_{i_0}\vert^2}{\vert \xi-z_{i}\vert^2}=
\frac{n-\mu_i}{\mu_i}\cdot  \frac{1}{(m-1)n-(n-1)}<\frac{1}{2\mu_i}.
$$
If $\mu_i=1$, this is a contradiction to our assumption that
$\vert \xi-z_{i_0}\vert=\max_{\{i\,\vert\,\mu_i=1\}} $ $\vert \xi-z_{i}\vert$.
If $\mu_i>1$, then
$$
\vert \xi-z_{i_0}\vert<1.
$$
Accordingly, it is sufficient to consider the case $\vert z_i\vert=1$ and $\eta_i>0$, $i\neq i_0$. Let $z_i:=e^{\i \gamma_i}$,~$i=1,2,\ldots,m$, $i\neq i_0$. By using an appropriate rotation with  center 0,
we can and will suppose that $\Im (\xi-z_{i_0})=0$ and $\xi-z_{i_0}<0$. Take an arbitrary $1\le i\le m$, $i\neq i_0$. Then, it follows from (\ref{kt-eq2}) that 0 lies in the interior of the triangle with vertices 1, $\psi_1:=(\xi-z_i)^2=c_i^2e^{\i 2\theta_i}$ and $\psi_2:=(\xi-z_i)^2\bar{z}_i=c_i^2e^{\i (2\theta_i-\gamma_i)}$, and thus,
\be\label{conv1}
\Im\psi_1\Im\psi_2<0.
\ee
Further, if $l\,:\,\Im z=\alpha\Re z+\beta$ is the line passing through $\psi_1$ and $\psi_2$, and $\psi$ is the intersection point
of $l$ with the real axis, then
$$
\psi=-\frac{\beta}{\alpha}=\frac{\Im \psi_2\Re \psi_1-\Im \psi_1\Re \psi_2}{\Im \psi_2-\Im \psi_1}<0,
$$
or equivalently,
\be\label{conv2}
\Re \psi_1<\frac{\Im \psi_1}{\Im \psi_2}\Re \psi_2.
\ee
Clearly, inequalities (\ref{conv1}) and (\ref{conv2}) are equivalent to
\begin{equation}\label{conv}
	\begin{aligned}
		\sin(2\theta_i)&\sin(2\theta_i-\gamma_i)<0,\\
		\cos(2\theta_i)&<\frac{\sin(2\theta_i)}{\sin(2\theta_i-\gamma_i)}\cos(2\theta_i-\gamma_i).
	\end{aligned}
\end{equation}
In view of inequalities (\ref{conv}), we have either $\theta_i\in (\pi/2,\pi)$ and $\gamma_i\in (0,\pi)$ 
(if $\sin(2\theta_i)<0$ and $\sin(\gamma_i)>0$) or $\theta_i\in (\pi,3\pi/2)$ and $\gamma_i\in (\pi,2\pi)$  (if $\sin(2\theta_i)>0$ and $\sin(\gamma_i)<0$). In both cases the points $\xi$,    $z_i$ and $\xi-z_i$ lie in the same open half-plane with respect to the real axis, for any $i=1,2,\ldots,m,~i\neq i_0$. Since $\vert z_{i_0}\vert <1$, this implies that $\hbox{rad}_c(\bar{z})<1$, which is a contradiction. Consequently, the maximum of $\vert z_{i_0}-\xi \vert$ is attained when all the points $z_i,~i=1,2,\ldots, m,$ lie on the unit circle $\partial D$. Applying  \cite[Theorem 1]{rub} to $z_{i_0}$, it follows that $\xi=0$ and the maximum is equal to 1. Finally, using again \cite[Theorem 1]{rub}, we conclude that this can happen only if $m=n$, $k=1$ and $p(z)=z^n-1$. Theorem~\ref{theorem2} is proved.
\hfill $\Box$


\end{document}